\documentclass[12pt,reqno]{amsart}
\usepackage{enumerate, latexsym, amsmath, amsfonts, amssymb, amsthm, graphicx, color}
\def\pmod #1{\ ({\rm{mod}}\ #1)}
\def\Z{\Bbb Z}

\def\Q{\Bbb Q}

\def\R{\Bbb R}

\def\F{\Bbb F}

\def\bg{\bigg}
\def\({\bg(}
\def\){\bg)}

\def\sgn{{\rm sgn}}

\def\sgn{{\rm sgn}}

\def\Gal{{\rm Gal}}

\def\ve{\varepsilon}

\def\Ack{\medskip\noindent {\bf Acknowledgments}}
\theoremstyle{plain}
\newtheorem{theorem}{Theorem}

\newtheorem{lemma}{Lemma}
\newtheorem{corollary}{Corollary}
\newtheorem{conjecture}{Conjecture}
\theoremstyle{definition}

\theoremstyle{remark}

\makeatletter
\@namedef{subjclassname@2010}{%
  \textup{2010} Mathematics Subject Classification}
\makeatother
 \vspace{4mm}
\begin{document}
\title[Determinants concerning Legendre symbols]{Determinants concerning Legendre symbols}

\author[H.-L. Wu]{Hai-Liang Wu}

\address {(Hai-Liang Wu) School of Science, Nanjing University of Posts and Telecommunications, Nanjing 210023, People's Republic of China}
\email{\tt whl.math@smail.nju.edu.cn}

\begin{abstract}
The evaluations of determinants with Legendre symbol entries have close relation with character sums over finite fields. Recently, Sun \cite{ffa1} posed some conjectures on this topic. In this paper, we prove some conjectures of Sun and also study some variants. For example, we show the following result:

Let $p=a^2+4b^2$ be a prime with $a,b$ integers and $a\equiv1\pmod4$. Then for the determinant
$$S(1,p):=\det\bigg[\(\frac{i^2+j^2}{p}\)\bigg]_{1\le i,j\le \frac{p-1}{2}},$$
the number $S(1,p)/a$ is an integral square, which confirms a conjecture posed by Cohen, Sun and Vsemirnov.
\end{abstract}

\thanks{
\newline \indent Supported by the National Natural Science
Foundation of China (Grant No. 11971222).}

\maketitle
\section{Introduction}	
\setcounter{lemma}{0}
\setcounter{theorem}{0}
\setcounter{corollary}{0}
\setcounter{remark}{0}
\setcounter{equation}{0}
\setcounter{conjecture}{0}
Given an $n\times n$ complex matrix $M=[a_{ij}]_{1\le i,j\le n}$, we often use $\det M$ or $|M|$ to denote the determinant of $M$. The evaluation of determinants with Legendre symbol entries is a classical topic in number theory and in finite fields. Krattenthaler's survey papers \cite{K1,K2} introduce many concrete examples and advanced techniques on determinant calculation.

Let $p$ be an odd prime and let $(\frac{\cdot }{p})$ denote the Legendre symbol. Carlitz \cite{carlitz} studied the following $(p-1)\times(p-1)$ matrix
$$D_p:=\bigg[\(\frac{i-j}{p}\)\bigg]_{1\le i,j\le p-1}.$$
He obtained that the characteristic polynomial of $D_p$ is precisely
$$\bigg{|}xI_{p-1}-D_p\bigg{|}=\(x^2-(-1)^{\frac{p-1}{2}}p\)^{\frac{p-3}{2}}\(x^2-(-1)^{\frac{p-1}{2}}\),$$
where $I_{p-1}$ is the $(p-1)\times(p-1)$ identity matrix.

Along this line, Chapman \cite{chapman} further investigated the following matrices:
$$C_p(x):=\left[x+\(\frac{i+j-1}{p}\)\right]_{1\le i,j\le\frac{p-1}{2}}$$
and
$$C_p^*(x):=\left[x+\(\frac{i+j-1}{p}\)\right]_{1\le i,j\le\frac{p+1}{2}},$$
where $x$ is a variable.
In the case $p\equiv 1\pmod 4$, let $\varepsilon_p>1$ and $h(p)$ be the fundamental unit and class number of the real quadratic field $\mathbb{Q}(\sqrt{p})$ respectively and let $\varepsilon_p^{h(p)}=a_p+b_p\sqrt{p}$ with $2a_p,2b_p\in\Z$. Chapman proved that
 	$$\det C_p(x)=
 \begin{cases}
 (-1)^{(p-1)/4}2^{(p-1)/2}(b_p-a_px)   &\mbox{if}\ p \equiv 1 \pmod 4,\\
 -2^{(p-1)/2}x &\mbox{if} \ p\equiv 3 \pmod 4,
 \end{cases}
 $$	
and that
	$$\det C_p^{*}(x)=
\begin{cases}
 (-1)^{(p-1)/4}2^{(p-1)/2}(pb_px-a_p)   &\mbox{if}\ p \equiv 1 \pmod 4,\\
-2^{(p-1)/2} &\mbox{if} \ p\equiv 3 \pmod 4.
\end{cases}
$$	
Moreover, Chapman \cite{evil} posed a conjecture concerning the determinant of the $\frac{p+1}{2}\times\frac{p+1}{2}$ matrix
$$C=\bigg[\(\frac{j-i}{p}\)\bigg]_{1\le i,j\le\frac{p+1}{2}}.$$
Due to the difficulty of the evaluation on this determinant, he called it ``evil" determinant. Finally this conjecture was confirmed completely by Vsemirnov \cite{M1,M2}.

Recently Sun \cite{ffa1} studied various determinants of matrices involving Legendre symbol entries. Let $p$ be a prime and $d$ be an integer with $p\nmid d$. Sun defined
$$S(d,p):=\det\bigg[\(\frac{i^2+dj^2}{p}\)\bigg]_{1\le i,j\le\frac{p-1}{2}}.$$
In the same paper, Sun also studied some properties of the above determinant. For example, he showed that
$-S(d,p)$ is always a quadratic residue modulo $p$ if $(\frac{d}{p})=1$ and that $S(d,p)=0$ if $(\frac{d}{p})=-1$.
Moreover, Sun posed the following conjecture:
\begin{conjecture}\label{Conjecture A}{\rm(Sun)}
Let $p\equiv3\pmod4$ be a prime. Then $-S(1,p)$ is an integral square.
\end{conjecture}
This conjecture was later confirmed by Alekseyev and Krachun by using some algebraic number theory. In the case $p\equiv1\pmod4$, Cohen, Sun and Vsemirnov also posed the following conjecture.
\begin{conjecture}\label{Conjecture B}{\rm(Cohen, Sun and Vsemirnov)}
Let $p=a^2+4b^2$ be a prime with $a,b$ integers and $a\equiv1\pmod4$. Then
$S(1,p)/a$ is an integral square.
\end{conjecture}
For example, if $p=5=1^2+4\times1^2$, then $S(1,5)=1=1\times1^2$. If $p=13=(-3)^2+4\times1^2$, then $S(1,13)=-27=-3\times3^2$.

As the first result of this paper, by considering some character sums over finite fields,
we confirm this conjecture and obtain the following result. For convenience, for each $d\in\Z$ we set
$$\ve(d)=\begin{cases}-1&\mbox{if}\ (\frac{d}{p})=1\ \text{and $d$ is not a biquadratic residue modulo $p$},\\1&\mbox{otherwise}\ .\end{cases} $$
\begin{theorem}\label{theorem A}
Let $p=a^2+4b^2$ be a prime with $a,b$ integers and $a\equiv1\pmod4$ and let $d$ be an integer. Then
$\ve(d)S(d,p)/a$ is an integral square. In particular, when $d=1$ the number $S(1,p)/a$ is an integral square.
\end{theorem}

Sun \cite{ffa1} also made the following conjecture.
\begin{conjecture}\label{Conjecture C}{\rm(Sun)}
Let $S^*(1,p)$ denote the determinant obtained from $S(1,p)$ via replacing the entries $(\frac{1^2+j^2}{p})$ $(j=1,\cdots,\frac{p-1}{2})$ in the first row by $(\frac{j}{p})$ $(j=1,\cdots,\frac{p-1}{2})$ respectively. Then $-S^*(1,p)$ is an integral square if $p\equiv1\pmod4$.
\end{conjecture}

As an application of Theorem \ref{theorem A}, we confirm this conjecture.
\begin{corollary}\label{corollary A}
Let $p\equiv1\pmod4$ be a prime. Then $-S^*(1,p)$ is an integral square.
\end{corollary}
For example, $S^*(1,5)=-1^2$, $S^*(1,13)=-3^2$ and $S^*(1,17)=-21^2$.

The proofs of our main results will be given at Sections 2.
\maketitle
\section{Proofs of the main results}
\setcounter{lemma}{0}
\setcounter{theorem}{0}
\setcounter{corollary}{0}
\setcounter{remark}{0}
\setcounter{equation}{0}
\setcounter{conjecture}{0}

We begin with the following permutation involving quadratic residues (readers may refer to \cite{Hou,ffa2} for details on the recent progress on permutations over finite fields). Let $p\equiv1\pmod4$ be a prime and let $d\in\Z$ with $(\frac{d}{p})=1$. If we write $p=2n+1$, then clearly the sequence
$$d\cdot1^2\ {\rm mod}\ p, \cdots, d\cdot n^2\ {\rm mod}\ p$$
is a permutation $\pi_p(d)$ of the sequence
$$1^2\ {\rm mod}\ p, \cdots, n^2\ {\rm mod}\ p.$$
Let $\sgn(\pi_p(d))$ be the sign of $\pi_p(d)$. We first have the following result:
\begin{lemma}\label{Lemma permutation}
Let $p\equiv1\pmod4$ be a prime, and let $d\in\Z$ be a quadratic residue modulo $p$. Then
$$\sgn(\pi_p(d))=\begin{cases}1&\mbox{if}\ $d$\ \text{is a biquadratic residue modulo $p$},\\-1&\mbox{otherwise}\ .\end{cases}$$
\end{lemma}
\begin{proof}
It is clear that
$$\sgn(\pi_p(d))\equiv\prod_{1\le i<j\le n}\frac{dj^2-di^2}{j^2-i^2}\pmod p.$$
By this we obtain
$$\sgn(\pi_p(d))\equiv (d^{\frac{p-1}{4}})^{n-1}\equiv d^{\frac{p-1}{4}}\pmod p.$$
This implies the desired result.
\end{proof}

We also need the following known result concerning eigenvalues of a matrix.
\begin{lemma}\label{Lemma eigenvalues}
Let $M$ be an $m\times m$ complex matrix. Let $\mu_1,\cdots,\mu_m$ be complex numbers, and let ${\bf u}_1,\cdots, {\bf u}_m$ be $m$-dimensional column vectors. Suppose that $M{\bf u}_k=\mu_k{\bf u}_k$ for each $1\le k\le m$ and that ${\bf u}_1,\cdots, {\bf u}_m$ are linear independent. Then $\mu_1,\cdots,\mu_m$ are exactly all the eigenvalues of $M$ (counting multiplicity).
\end{lemma}

Before the proof of Theorem \ref{theorem A}, we first introduce some notations. In the remaining part of this section, we let $p=a^2+4b^2$ be a prime with $a,b\in\Z$ and $a\equiv1\pmod4$, and let $n=\frac{p-1}{2}$. In addition, we let $\chi(\Z/p\Z)$ denote the group of all multiplicative characters on the finite field $\Z/p\Z=\F_p$, and let $\chi_p$ be a generator of $\chi(\Z/p\Z)$, i.e.,
$$\chi(\Z/p\Z)=\{\chi_p^k: k=1,2,\cdots,p-1\}.$$
Readers may refer to \cite[Chapter 8]{classical} for a detailed introduction on characters on finite fields. Also, given any matrix $M$, the symbol $M^T$ denotes the transpose of $M$.

Now we are in a position to prove our first theorem.

{\bf Proof of Theorem \ref{theorem A}.} Throughout this proof, we define
$$M_p:=\bigg[\(\frac{i^2+j^2}{p}\)\bigg]_{1\le i,j\le n}.$$
We first determine all the eigenvalues of $M_p$. For $k=1,2,\cdots,n$, we let
\begin{equation}\label{Eq. eigenvalues lambda k}
\lambda_k:=\sum_{1\le j\le n}\(\frac{1+j^2}{p}\)\chi_p^k(j^2).
\end{equation}
We claim that $\lambda_1,\cdots,\lambda_n$ are exactly all the eigenvalues of $M_p$ (counting multiplicity). In fact, for any $1\le i,k\le n$ we have
\begin{align*}
\sum_{1\le j\le n}\(\frac{i^2+j^2}{p}\)\chi_p^k(j^2)&=\sum_{1\le j\le n}\(\frac{1+j^2/i^2}{p}\)\chi_p^k(j^2/i^2)\chi^k_p(i^2)\\
&=\sum_{1\le j\le n}\(\frac{1+j^2}{p}\)\chi_p^k(j^2)\chi_p^k(i^2)=\lambda_k\chi_p^k(i^2).
\end{align*}
This implies that for each $k=1,\cdots,n$, we have
$$M_p{\bf v}_k=\lambda_k{\bf v}_k,$$
where
$${\bf v}_k:=(\chi_p^k(1^2),\chi_p^k(2^2),\cdots,\chi_p^k(n^2))^T.$$
Since
$$\left| \begin{array}{cccccccc}
\chi_p^1(1^2) & \chi_p^2(1^2) & \ldots  & \chi_p^n(1^2)  \\
\chi_p^1(2^2) & \chi_p^2(2^2) &  \ldots  &  \chi_p^n(2^2)\\
\vdots & \vdots & \ddots  & \vdots    \\
\chi_p^1(n^2)& \chi_p^2(n^2) & \ldots &  \chi_p^n(n^2)\\
\end{array} \right|=\pm\prod_{1\le i<j\le n}\(\chi_p(j^2)-\chi_p(i^2)\)\ne0,$$
the vectors ${\bf v}_1,\cdots,{\bf v}_n$ are linear independent. By Lemma \ref{Lemma eigenvalues} our claim holds. Hence we have
\begin{equation}\label{Eq. S(1,p)}
S(1,p)=\det M_p=\prod_{1\le k\le n}\lambda_k=\prod_{1\le k\le n}\(\sum_{1\le j\le n}\(\frac{1+j^2}{p}\)\chi_p^k(j^2)\).
\end{equation}
Now we turn to the last product. When $k=n$, by \cite[Exercise 8, p. 63]{classical} we have
\begin{equation}\label{Eq. A in the proof of Theorem A}
\lambda_n=\sum_{1\le j\le n}\(\frac{1+j^2}{p}\)\chi_p^n(j^2)=\sum_{1\le j\le n}\(\frac{1+j^2}{p}\)=-1.
\end{equation}
When $k=n/2$, by \cite[Theorem 6.2.9]{BEW} we have
\begin{equation}\label{Eq. B in the proof of Theorem A}
\lambda_{n/2}=\sum_{1\le j\le n}\(\frac{1+j^2}{p}\)\chi_p^{n/2}(j^2)=\sum_{1\le j\le n}\(\frac{1+j^2}{p}\)\(\frac{j}{p}\)=-a.
\end{equation}
In addition, the product
$$\prod_{1\le k\le n,k\ne n,n/2}\(\sum_{1\le j\le n}\(\frac{1+j^2}{p}\)\chi_p^k(j^2)\)$$
is equal to
\begin{align*}
\prod_{k=1}^{\frac{p-5}{4}}\(\sum_{1\le j\le n}\(\frac{1+j^2}{p}\)\chi_p^k(j^2)\)\(\sum_{1\le j\le n}\(\frac{1+j^2}{p}\)\chi_p^{-k}(j^2)\)=\prod_{k=1}^{\frac{p-5}{4}}\lambda_k\overline{\lambda_k},
\end{align*}
where $\overline{\lambda_k}$ is the complex conjugation of $\lambda_k$.
As $M_p$ is a real symmetric matrix, every eigenvalue $\lambda_k$ of $M_p$ is real. Hence
\begin{equation}\label{Eq. C in the proof of Theorem A}
\prod_{1\le k\le n,k\ne n,n/2}\(\sum_{1\le j\le n}\(\frac{1+j^2}{p}\)\chi_p^k(j^2)\)=\(\prod_{k=1}^{\frac{p-5}{4}}\lambda_k\)^2.
\end{equation}
Now we claim that $\prod_{k=1}^{\frac{p-5}{4}}\lambda_k\in\Z$. In fact, let $K$ be the cyclotomic field $\Q(\zeta_{p-1})$, where $\zeta_{p-1}$ is a primitive $(p-1)$-th root of unity.
Clearly we have $\lambda_k\in K\cap\R$. Given any isomorphism $\sigma_l$ in the Galois group $\Gal(K/\Q)$ with $\sigma_l(\zeta_{p-1})=\zeta_{p-1}^l$, we have
$$\sigma_{l}\(\prod_{k=1}^{\frac{p-5}{4}}\lambda_k\)
=\prod_{k=1}^{\frac{p-5}{4}}\(\sum_{1\le j\le n}\(\frac{1+j^2}{p}\)\chi_p^{2kl}(j)\).$$
One can easily verify that for any $1\le k,k_1,k_2\le\frac{p-5}{4}$ with $k_1\ne k_2$, all of the following are satisfied:

(1) $2kl\not\equiv 0\pmod{p-1}$,

(2) $2kl\not\equiv n\pmod{p-1}$,

(3) $2k_1l\not\equiv\pm2k_2l\pmod{p-1}$.

Hence for each $1\le k\le \frac{p-5}{4}$ there is a unique $1\le r\le\frac{p-5}{4}$ such that $2kl\equiv\pm 2r\pmod{p-1}$. As $\lambda_k\in K\cap\R$ and $(K\cap\R)/\Q$ is a Galois extension, we have
$$\(\sum_{1\le j\le n}\(\frac{1+j^2}{p}\)\chi_p^{2kl}(j)\)=\sigma_l(\lambda_k)\in\R$$
and hence
$$\(\sum_{1\le j\le n}\(\frac{1+j^2}{p}\)\chi_p^{2kl}(j)\)=\(\sum_{1\le j\le n}\(\frac{1+j^2}{p}\)\chi_p^{-2kl}(j)\).$$
By this and the above, we obtain
$$\prod_{k=1}^{\frac{p-5}{4}}\(\sum_{1\le j\le n}\(\frac{1+j^2}{p}\)\chi_p^{2kl}(j)\)=
\prod_{r=1}^{\frac{p-5}{4}}\(\sum_{1\le j\le n}\(\frac{1+j^2}{p}\)\chi_p^{2r}(j)\),$$
i.e.,
$$\sigma_l\(\prod_{k=1}^{\frac{p-5}{4}}\lambda_k\)=\prod_{k=1}^{\frac{p-5}{4}}\lambda_k.$$
Hence $\prod_{k=1}^{\frac{p-5}{4}}\lambda_k\in\Q$ by the Galois theory. As $\prod_{k=1}^{\frac{p-5}{4}}\lambda_k$ is an algebraic integer, we further obtain that $\prod_{k=1}^{\frac{p-5}{4}}\lambda_k\in\Z$. Hence by (\ref{Eq. S(1,p)})--(\ref{Eq. C in the proof of Theorem A}) we obtain that $S(1,p)/a$ is an integral square.

Now we consider $S(d,p)$. If $p\mid d$, then clearly $S(d,p)=0$. If $(\frac{d}{p})=-1$, then by \cite[Theorem 1.2]{ffa1} we know that $S(d,p)=0$. Suppose now that $d$ is a quadratic residue modulo $p$. Then clearly we have
$$S(d,p)=\sgn(\pi_p(d))S(1,p).$$
Now our desired result follows from Lemma \ref{Lemma permutation}.\qed

We now prove our next result.

{\bf Proof of Corollary \ref{corollary A}.} By \cite[Theorem 6.2.9]{BEW} for any $1\le i,j\le n$ we have
$$\sum_{1\le i\le n}\(\frac{i^2+j^2}{p}\)\(\frac{i}{p}\)=-a\(\frac{j}{p}\)$$
and hence
\begin{equation}\label{Eq. A in the proof of corollary A}
-\sum_{2\le i\le n}\(\frac{i^2+j^2}{p}\)\(\frac{i}{p}\)-a\(\frac{j}{p}\)=\(\frac{1+j^2}{p}\).
\end{equation}
By this we have
$$S^*(1,p)=\frac{-1}{a}\left| \begin{array}{cccccccc}
-a(\frac{1}{p}) & -a(\frac{2}{p}) & \ldots  & -a(\frac{n}{p})  \\
(\frac{2^2+1^2}{p}) & (\frac{2^2+2^2}{p}) &  \ldots  &  (\frac{2^2+n^2}{p})\\
\vdots & \vdots & \ddots  & \vdots    \\
(\frac{n^2+1^2}{p})& (\frac{n^2+2^2}{p}) & \ldots &  (\frac{n^2+n^2}{p})\\
\end{array} \right|=-S(1,p)/a.$$
The last equality follows from (\ref{Eq. A in the proof of corollary A}). Now our desired result follows from Theorem \ref{theorem A}.\qed

\Ack\ This research was supported by the National Natural Science Foundation of
China (Grant No. 11971222).


\begin{thebibliography}{99}
\bibitem{BEW}  B. C. Berndt, R. J. Evans and K. S. Williams, Gauss and Jacobi Sums,
Wiley, New York, 1998.
\bibitem{carlitz}  L. Carlitz, Some cyclotomic matrices, Acta Arith. 5 (1959) 293--308.
\bibitem{chapman} R. Chapman, Determinants of Legendre symbol matrices, Acta Arith. 115 (2004) 231--244.
\bibitem{evil} R. Chapman, My evil determinant problem, preprint, December 12, 2012, available at http://empslocal.ex.ac.uk/people/staff/rjchapma/etc/evildet.pdf.
\bibitem{Hou} X.-D. Hou, Permutation polynomials over finite fields, A survey of recent advances, Finite Field Appl. 32 (2015), 82--119.
\bibitem{classical} K. Ireland, M. Rosen, A Classical Introduction to Modern Number Theory, second ed., Grad. Texts Math., vol. 84, Springer, New York, 1990.
\bibitem{K1} C. Krattenthaler, Advanced determinant calculus, S W min. Lothar. Comb. 42 (1999)
B42q.
\bibitem{K2}  C. Krattenthaler, Advanced determinant calculus: a complement, Linear Algebra
Appl. 411 (2005) 68--166.
\bibitem{ffa1} Z.-W. Sun, On some determinants with Legendre symbol entries, Finite Fields Appl. 56 (2019) 285--307.
\bibitem{ffa2} Z.-W. Sun, Quadratic residues and related permutations and identities, Finite Fields Appl. 59 (2019) 246--283.
\bibitem{M1} M. Vsemirnov, On the evaluation of R. Chapman's ``evil determinant", Linear Algebra Appl. 436 (2012) 4101--4106.
\bibitem{M2} M. Vsemirnov, On R. Chapman's ``evil determinant": case $p\equiv1\pmod 4$, Acta Arith. 159 (2013)
 331--344.
\end{thebibliography}
\end{document}